\documentclass[11pt]{amsart}

\usepackage{amssymb,amsmath,amsthm}
\usepackage{enumerate}
\usepackage{booktabs}
\usepackage{tikz}
\usetikzlibrary{arrows.meta,calc,positioning,decorations.pathreplacing}
\usepackage[colorlinks=true,linkcolor=blue,citecolor=blue,urlcolor=blue]{hyperref}

\newtheorem{theorem}{Theorem}[section]
\newtheorem{lemma}[theorem]{Lemma}
\newtheorem{proposition}[theorem]{Proposition}

\newtheorem{conjecture}[theorem]{Conjecture}
\newtheorem{mainthm}{Theorem}

\newtheorem{maincor}[mainthm]{Corollary}
\theoremstyle{definition}
\newtheorem{definition}[theorem]{Definition}
\newtheorem{example}[theorem]{Example}
\newtheorem{problem}{Problem}
\theoremstyle{remark}
\newtheorem{remark}[theorem]{Remark}

\newcommand{\I}{I}
\newcommand{\eps}{\varepsilon}
\newcommand{\Int}{\operatorname{Int}}
\newcommand{\Lat}{\Lambda}
\newcommand{\Leb}{\operatorname{Leb}}

\begin{document}

\title[Connected $\omega$-limit sets of Cournot maps]{Connected $\omega$-limit sets of Cournot maps:\\
periodic orbits, realization and almost every point}

\author{Juan Luis Garc\'{\i}a Guirao}
\address{Departamento de Matem\'atica Aplicada y Estad\'{\i}stica,
Universidad Polit\'ecnica de Cartagena, Hospital de Marina,
30203 Cartagena (Regi\'on de Murcia), Spain}
\email{juan.garcia@upct.es}


\begin{abstract}
Let $F(x,y)=(f_2(y),f_1(x))$ be a Cournot map of the unit square, $\varphi=f_2\circ f_1$ and $\chi=f_1\circ f_2$.
Linero Bas and Mu\~noz Guillermo proved that an $\omega$-limit set of $F$ with nonempty interior is connected if
and only if \textup{(C1)} its two projections are connected and \textup{(C2)}
$\omega_{F^2}(z)\cap\omega_{F^2}(F(z))\neq\emptyset$, and conjectured that this holds for every $\omega$-limit set.
We prove the conjecture for every point $z=(x,y)$ such that $\omega_\varphi(x)$ or $\omega_\chi(y)$ is finite, and
for Lebesgue almost every point $z$ of the square whenever $\varphi$ is metrically regular (every metric attractor
is a periodic orbit, a solenoid, or a cycle of intervals carrying a well-behaved absolutely continuous invariant
measure). Over a periodic orbit the $\omega$-limit set is a \emph{comb}, whose sections are rigid; when one
projection is connected it is a point, a \emph{lattice} $(P\times J)\cup(K\times Q)$ with $P\subset K$, $Q\subset J$
and $\#P=\#Q$, or an \emph{alternating comb} of $2m$ disjoint segments. Conversely, every lattice and every
alternating comb is an $\omega$-limit set of a Cournot map with piecewise linear reaction functions; the
construction rests on a lemma producing exact compositions of piecewise linear maps with prescribed values. This
answers, for this class, the question of Balibrea, Linero and C\'anovas of which nowhere dense sets are
$\omega$-limit sets of Cournot maps. For instance, the letter H is not an $\omega$-limit set at any such point, and
for metrically regular $\varphi$ the points whose $\omega$-limit set is connected, infinite, nowhere dense and not a
lattice form a Lebesgue null set.
\end{abstract}

\maketitle

\section{Introduction}\label{sec:intro}

A \emph{Cournot map} (or \emph{antitriangular map}) of the square $\I^2$, $\I=[0,1]$, is a continuous map
\[
F(x,y)=\bigl(f_2(y),\,f_1(x)\bigr),\qquad f_1,f_2\colon\I\to\I \text{ continuous}.
\]
These maps model the Cournot duopoly, in which each firm chooses its output as a reaction to the output of its
rival in the previous period \cite{Co,DM,Pu}; $f_1$ and $f_2$ are the \emph{reaction functions}. Their dynamics is
governed by the two compositions
\[
\varphi=f_2\circ f_1,\qquad \chi=f_1\circ f_2,
\]
since $F^{2n}(x,y)=(\varphi^n(x),\chi^n(y))$ for all $n\ge0$ \cite{BCL01,BLC,DM}. In particular $F^2$ is the
product map $\varphi\times\chi$, and the attractors of $F^2$ are built from attractors of $\varphi$ and $\chi$;
this observation is the basis of the study of multistability in duopoly games \cite{ABG,DM}.

For interval maps the topology of $\omega$-limit sets is completely understood: a nonempty closed set $A\subset\I$
is an $\omega$-limit set of some continuous map $\I\to\I$ if and only if $A$ is nowhere dense or a finite union of
nondegenerate closed intervals \cite{ABCP,Sh}. Balibrea, Linero and C\'anovas \cite{BLC} proved that every
$\omega$-limit set of a Cournot map is finite, infinite and nowhere dense, or a finite union of nondegenerate
periodic rectangles with pairwise disjoint interiors; they observed that, unlike the interval case, a nowhere
dense $\omega$-limit set can be connected, and left open the question of \emph{which} nowhere dense sets occur.
Linero Bas and Mu\~noz Guillermo \cite{LM} described completely the $\omega$-limit sets with nonempty interior,
constructed connected $\omega$-limit sets with empty interior (circles, chains of circles, lattices
$(\I\times P)\cup(P\times\I)$), and characterized connectedness when the interior is nonempty.

\begin{theorem}[{Linero Bas and Mu\~noz Guillermo \cite[Prop.~5]{LM}}]\label{thm:LM}
Let $F$ be a Cournot map and $z\in\I^2$, and let $\pi_1,\pi_2$ be the coordinate projections.
If $\omega_F(z)$ is connected, then
\begin{enumerate}[\upshape(C1)]
\item $\pi_1(\omega_F(z))$ and $\pi_2(\omega_F(z))$ are connected, and
\item $\omega_{F^2}(z)\cap\omega_{F^2}(F(z))\neq\emptyset$.
\end{enumerate}
Conversely, if $\Int\omega_F(z)\neq\emptyset$, then \textup{(C1)} and \textup{(C2)} imply that $\omega_F(z)$ is connected.
\end{theorem}

\begin{conjecture}[{\cite[\S6]{LM}}]\label{conj:LM}
For every Cournot map $F$ and every $z\in\I^2$, $\omega_F(z)$ is connected if and only if \textup{(C1)} and
\textup{(C2)} hold.
\end{conjecture}

By Theorem~\ref{thm:LM} only sufficiency for $\omega$-limit sets with empty interior is open. We address this
problem from three sides: we prove the conjecture for all points lying over a periodic orbit, we determine exactly
which sets occur there, and we prove the conjecture for Lebesgue almost every point under a metric regularity
assumption.

\subsection*{Points over periodic orbits}

\begin{mainthm}\label{thm:A}
Let $F(x,y)=(f_2(y),f_1(x))$ be a Cournot map and let $z=(x,y)\in\I^2$ be such that $\omega_{f_2\circ f_1}(x)$
or $\omega_{f_1\circ f_2}(y)$ is finite. Then $\omega_F(z)$ is connected if and only if \textup{(C1)} and
\textup{(C2)} hold.
\end{mainthm}

If $\omega_\varphi(x)$ is finite, $\omega_F(z)$ lies in finitely many vertical and horizontal lines, so it has
empty interior; Theorem~\ref{thm:A} concerns only the open case of the conjecture. It follows from a finer
description. Suppose that $\omega_\varphi(x)=P$ is a periodic orbit $\{p_0,\dots,p_{m-1}\}$ with
$\varphi(p_i)=p_{i+1}$ (indices modulo $m$), labelled so that $\varphi^{km+i}(x)\to p_i$. Put
\[
M_i=\omega_{\chi^m}\bigl(\chi^i(y)\bigr),\quad q_i=f_1(p_i),\quad N_i=f_2(M_i),\quad Q=\{q_0,\dots,q_{m-1}\}.
\]
We show in Proposition~\ref{prop:comb} that
\begin{equation}\label{eq:comb}
\omega_F(z)=\bigcup_{i=0}^{m-1}V_i\ \cup\ \bigcup_{j=0}^{m-1}H_j,\qquad V_i=\{p_i\}\times M_i,\quad H_j=N_j\times\{q_j\},
\end{equation}
a union of $m$ \emph{vertical} and $m$ \emph{horizontal} pieces that we call a \emph{comb}. The sections $M_i$
are rigid: they are all finite of the same cardinality, all infinite and nowhere dense, or all finite unions of
the same number of intervals (Proposition~\ref{prop:rigidity}). This contrasts with triangular maps
$(x,y)\mapsto(f(x),g(x,y))$, where consecutive sections over a periodic orbit can be a Cantor set and a Cantor set
together with an interval \cite{Gu26}.

\begin{mainthm}\label{thm:B}
In the setting above, assume that $\pi_2(\omega_F(z))$ is connected. Then exactly one of the following holds.
\begin{enumerate}[\upshape(I)]
\item \textup{(Lattice)} There is a closed interval $J$, possibly degenerate, with $M_0=\dots=M_{m-1}=J$.
Then $Q\subset J$, $P\subset f_2(J)$, every vertical piece meets every horizontal piece,
\[
\omega_F(z)=(P\times J)\cup\bigl(f_2(J)\times Q\bigr)
\]
is connected, and \textup{(C2)} holds.
\item \textup{(Alternating comb)} $m$ is even and there are nondegenerate closed intervals $J_0,J_1$ with
$J_0\cap J_1=\{c\}$, $\chi(J_0)=J_1$, $\chi(J_1)=J_0$, such that the sections alternate,
$M_i=J_{(i+k)\bmod 2}$ for some $k$. Then \textup{(C2)} fails, and $\omega_F(z)$ is not connected.
\end{enumerate}
\end{mainthm}

\subsection*{Realization}
The next result determines exactly which sets occur in Theorem~\ref{thm:B}.

\begin{definition}\label{def:lattice}
Let $m\ge1$. A \emph{lattice of order $m$} is a set $\Lat=(P\times J)\cup(K\times Q)\subset\I^2$, where $J,K\subset\I$
are nondegenerate compact intervals and $P\subset K$, $Q\subset J$ are sets with exactly $m$ points each.
\end{definition}

A lattice is connected: the vertical segment $\{p\}\times J$ meets the horizontal segment $K\times\{q\}$ at $(p,q)$.
For $m=1$ the lattices are the crosses, the letters T and L and their rotations (Figure~\ref{fig:m1}).

\begin{definition}\label{def:alternating}
Let $m\ge2$ be even. An \emph{alternating comb of order $m$} is a set
\[
\Gamma=\bigcup_{i=0}^{m-1}\bigl(\{p_i\}\times J_{i}\bigr)\ \cup\ \bigcup_{i=0}^{m-1}\bigl(K_{i}\times\{q_i\}\bigr)
\qquad(\text{indices of }J,K\text{ modulo }2),
\]
where $J_0,J_1$ are nondegenerate compact intervals with $J_0\cap J_1=\{c\}$, $K_0,K_1$ are nondegenerate compact
intervals with $K_0\cap K_1=\{d\}$, and for some $s\in\{0,1\}$ the points
$p_i\in K_{i+s+1}\setminus\{d\}$ and $q_i\in J_{i+s}\setminus\{c\}$ $(0\le i<m)$
are such that $p_0,\dots,p_{m-1}$ are pairwise distinct and $q_0,\dots,q_{m-1}$ are pairwise distinct.
\end{definition}

The $2m$ segments of an alternating comb are pairwise disjoint (Lemma~\ref{lem:disjoint}), so $\Gamma$ is not
connected although its projections are intervals (Figure~\ref{fig:altcomb}).

\begin{mainthm}\label{thm:C}
Let $L\subset\I^2$ be a set with more than one point.
\begin{enumerate}[\upshape(1)]
\item $L=\omega_F(z)$ for a Cournot map $F$ and a point $z=(x,y)$ with $\omega_{f_2\circ f_1}(x)$ finite and $L$
connected if and only if $L$ is a lattice.
\item $L=\omega_F(z)$ for a Cournot map $F$ and a point $z=(x,y)$ with $\omega_{f_2\circ f_1}(x)$ finite,
$\pi_2(L)$ connected and $L$ not connected if and only if $L$ is an alternating comb.
\end{enumerate}
In both cases the reaction functions can be chosen piecewise linear, with $f_2\circ f_1$ and $f_1\circ f_2$
topologically exact on the relevant intervals, and $z$ with $x$ periodic.
\end{mainthm}

Thus the $\omega$-limit sets $\omega_F(z)$ with $\omega_\varphi(x)$ finite and $\pi_2(\omega_F(z))$ connected are
exactly the points, the lattices and the alternating combs (Table~\ref{tab:summary}). By the symmetry
$(x,y)\mapsto(y,x)$ (Remark~\ref{rem:swap}), which maps lattices to lattices, part (1) also holds with
$\omega_{f_1\circ f_2}(y)$ in place of $\omega_{f_2\circ f_1}(x)$.

\begin{maincor}\label{cor:D}
Let $F$ be a Cournot map and $z=(x,y)$ a point such that $\omega_{f_2\circ f_1}(x)$ or $\omega_{f_1\circ f_2}(y)$ is
finite. If $\omega_F(z)$ is connected and not a point, then it has finitely many maximal vertical segments and the
same number of maximal horizontal segments; all vertical segments have the same projection, all horizontal
segments have the same projection, and every vertical segment meets every horizontal one. In particular the
letter H, $(\{0,1\}\times\I)\cup(\I\times\{\frac12\})$, and the letters E and F are not $\omega$-limit sets of such
points.
\end{maincor}

The letter H is connected, has connected projections and is a finite union of segments, like a lattice; the
obstruction is dynamical. The proof of Theorem~\ref{thm:C} rests on a lemma (Lemma~\ref{lem:good} and
Proposition~\ref{prop:exact}) producing piecewise linear maps between intervals, with prescribed values at
finitely many points, whose compositions are topologically exact.

\subsection*{Almost every point}
For points in general position both $\omega_\varphi(x)$ and $\omega_\chi(y)$ are usually infinite, and
Theorem~\ref{thm:A} does not apply. We show that, nevertheless, the conjecture holds for Lebesgue almost every
point when $\varphi$ is \emph{metrically regular} in the sense of Definition~\ref{def:regular}: Lebesgue almost
every $x$ has $\omega_\varphi(x)$ equal to one of finitely many attractors, each of which is a periodic orbit, a
solenoid, or a cycle of intervals carrying an absolutely continuous invariant probability equivalent to Lebesgue
measure and exact up to a finite cycle. By the work of Blokh and Lyubich, van Strien and Vargas, Ledrappier and
others \cite{BL91,BPP,Le81,vSV}, this describes the typical behaviour of smooth interval maps with non-flat
critical points, apart from wild attractors \cite{BKNS} and interval attractors without such a measure.
Hyperbolic maps, which are dense \cite{KSvS}, are metrically regular; so are many maps with chaotic attractors,
for instance $f_1=f_2=4t(1-t)$, and piecewise linear examples with coexisting periodic and chaotic attractors
(Example~\ref{ex:mixed}).

\begin{mainthm}\label{thm:E}
Let $F(x,y)=(f_2(y),f_1(x))$ be a Cournot map whose reaction functions are nonsingular
(Definition~\ref{def:nonsingular}), and assume that $\varphi=f_2\circ f_1$ is metrically regular. Then for Lebesgue
almost every $z=(x,y)\in\I^2$ at least one of the following holds:
\begin{enumerate}[\upshape(a)]
\item $\omega_\varphi(x)$ or $\omega_\chi(y)$ is finite, and $\omega_F(z)$ is a comb as in \eqref{eq:comb} or its
image under $(x,y)\mapsto(y,x)$;
\item $\pi_1(\omega_F(z))$ is not connected;
\item $\Int\omega_F(z)\neq\emptyset$.
\end{enumerate}
Consequently Conjecture~\ref{conj:LM} holds for Lebesgue almost every $z\in\I^2$, and for almost every $z$, if
$\omega_F(z)$ is connected, then it is a point, a lattice, or a finite union of periodic rectangles.
\end{mainthm}

In particular, for metrically regular $\varphi$ the connected $\omega$-limit sets with empty interior that are
neither points nor lattices, such as the circles of \cite{LM} or the letter H, can only occur on a Lebesgue null set of points. In the
language of duopoly games: for almost every pair of initial outputs, the long-run regime of the market is either a
lattice-type regime, in which one firm settles on a cycle while the other fluctuates on an interval, or a regime
of positive area, or it splits into separated pieces.

\subsection*{Organization}
Section~\ref{sec:prelim} collects preliminaries. Section~\ref{sec:comb} proves the comb structure and the rigidity of
sections, Section~\ref{sec:prop} the propagation lemma, and Section~\ref{sec:proofs} Theorems~\ref{thm:A}
and~\ref{thm:B}. Section~\ref{sec:lemma} contains the construction lemma, and Sections~\ref{sec:lattices}
and~\ref{sec:alternating} prove Theorem~\ref{thm:C} and Corollary~\ref{cor:D}. Theorem~\ref{thm:E} is proved in
Section~\ref{sec:ae}. Section~\ref{sec:examples} contains examples and Section~\ref{sec:problems} open problems.

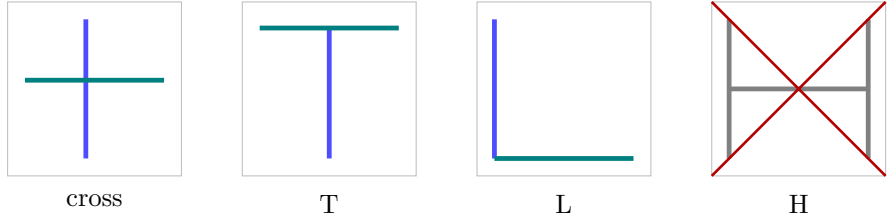
\begin{figure}[t]
\centering
\begin{tikzpicture}[scale=2.3]
  \foreach \ox/\lab in {0/{cross},1.35/{T},2.7/{L},4.05/{H}}{
    \draw[gray!50] (\ox,0) rectangle (\ox+1,1);
    \node[below] at (\ox+0.5,-0.05) {\small \lab};}
  \draw[line width=1.8pt,blue!70] (0.45,0.1)--(0.45,0.9);
  \draw[line width=1.8pt,teal] (0.1,0.55)--(0.9,0.55);
  \draw[line width=1.8pt,blue!70] (1.85,0.1)--(1.85,0.85);
  \draw[line width=1.8pt,teal] (1.45,0.85)--(2.25,0.85);
  \draw[line width=1.8pt,blue!70] (2.8,0.1)--(2.8,0.9);
  \draw[line width=1.8pt,teal] (2.8,0.1)--(3.6,0.1);
  \draw[line width=1.8pt,gray] (4.15,0.1)--(4.15,0.9);
  \draw[line width=1.8pt,gray] (4.95,0.1)--(4.95,0.9);
  \draw[line width=1.8pt,gray] (4.15,0.5)--(4.95,0.5);
  \draw[red!70!black,line width=1pt] (4.05,0)--(5.05,1); \draw[red!70!black,line width=1pt] (4.05,1)--(5.05,0);
\end{tikzpicture}
\caption{The lattices of order $1$ are crosses, T's and L's (and their rotations): one vertical segment
$\{p\}\times J$ and one horizontal segment $K\times\{q\}$ with $p\in K$, $q\in J$. All of them are $\omega$-limit
sets of Cournot maps (Theorem~\ref{thm:C}). The letter H is not an $\omega$-limit set at any point with
$\omega_{f_2\circ f_1}(x)$ or $\omega_{f_1\circ f_2}(y)$ finite (Corollary~\ref{cor:D}), and for metrically regular
maps it is the $\omega$-limit set of at most a null set of points (Theorem~\ref{thm:E}).}
\label{fig:m1}
\end{figure}

\section{Preliminaries}\label{sec:prelim}

Let $X$ be a compact metric space and $\psi\colon X\to X$ continuous. For every $u\in X$ the set $\omega_\psi(u)$
is nonempty, closed and strongly invariant, $\psi(\omega_\psi(u))=\omega_\psi(u)=\omega_\psi(\psi(u))$, and, for every
$n\ge1$ \cite{BC},
\begin{equation}\label{eq:split}
\omega_\psi(u)=\bigcup_{j=0}^{n-1}\omega_{\psi^n}\bigl(\psi^j(u)\bigr).
\end{equation}
For interval maps, $\omega_\psi(u)$ is finite if and only if $u$ is asymptotically periodic, and then
$\omega_\psi(u)$ is a periodic orbit \cite{BC}.

The map $\psi$ is \emph{transitive} if for all nonempty open $U,V\subset X$ there is $n\ge1$ with
$\psi^n(U)\cap V\neq\emptyset$, \emph{mixing} if this holds for all large $n$, and \emph{(topologically) exact}, or
locally eventually onto, if for every nonempty open $U$ there is $n$ with $\psi^n(U)=X$. Exact maps are mixing, and
every power of an exact map is exact. If $X$ has no isolated points, $\psi$ is transitive whenever some $u$
satisfies $\omega_\psi(u)=X$, and if $\psi$ is transitive the points $u$ with $\omega_\psi(u)=X$ form a residual set
\cite{KS97}. A mixing map $\psi$ of a compact interval $L=[l,l']$ has the following property \cite{Ru}: for every
nondegenerate interval $V\subset L$ and every $\delta>0$ there is $N$ such that $\psi^n(V)\supset[l+\delta,l'-\delta]$
for all $n\ge N$.

\begin{lemma}\label{lem:noinvariant}
Let $\psi$ be a transitive map of a nondegenerate compact interval $J$, and let $A\subset J$ be closed, with
$\psi(A)\subset A$ and nonempty interior. Then $A=J$.
\end{lemma}

\begin{proof}
If $A\neq J$, then $U=\Int A$ and $V=J\setminus A$ are nonempty and open in $J$. By transitivity
$\psi^n(U)\cap V\neq\emptyset$ for some $n$, but $\psi^n(U)\subset A$.
\end{proof}

\begin{theorem}[{\cite{BM}; see also \cite{BLC,KS97}}]\label{thm:BM}
Let $\psi$ be a transitive map of a compact interval $J=[a,b]$. Then either $\psi$ is mixing, or there is
$c\in(a,b)$ such that, with $J_0=[a,c]$ and $J_1=[c,b]$, $\psi(J_0)=J_1$, $\psi(J_1)=J_0$ and $\psi^2|_{J_k}$
is mixing for $k=0,1$. In the second case $\psi(c)=c$.
\end{theorem}

The last assertion holds because $\psi(c)\in\psi(J_0)\cap\psi(J_1)=\{c\}$. A \emph{cycle of intervals} of an interval
map $\psi$ is a finite union $K=K_0\cup\dots\cup K_{r-1}$ of pairwise disjoint nondegenerate compact intervals with
$\psi(K_i)=K_{i+1}$ (indices modulo $r$); it is \emph{transitive} if $\psi|_K$ is. Then $\psi^r|_{K_0}$ is transitive,
and by Theorem~\ref{thm:BM} the interval $K_0$ is either itself, or the union of its two halves, a nondegenerate
interval $L$ with $\psi^{2r}(L)=L$ and $\psi^{2r}|_L$ mixing. The same holds for every component $K_i$, so $K$ is the
union of finitely many nondegenerate intervals $L$ on which $\psi^{2r}$ is mixing. An $\omega$-limit set of an interval
map with nonempty interior is a transitive cycle of intervals \cite{ABCP,Sh}.

For a Cournot map we write $\varphi=f_2\circ f_1$ and $\chi=f_1\circ f_2$. Then
\begin{equation}\label{eq:semiconj}
f_1\circ\varphi=\chi\circ f_1,\qquad f_2\circ\chi=\varphi\circ f_2 .
\end{equation}

\begin{proposition}[{\cite[Prop.~7]{BLC}}]\label{prop:BLC}
For every $z=(x,y)\in\I^2$:
\begin{enumerate}[\upshape(a)]
\item $\omega_F(z)=\omega_{F^2}(z)\cup\omega_{F^2}(F(z))$ and $\omega_{F^2}(F(z))=F(\omega_{F^2}(z))$;
\item $\pi_1(\omega_F(z))=\omega_\varphi(x)\cup\omega_\varphi(f_2(y))$ and
$\pi_2(\omega_F(z))=\omega_\chi(y)\cup\omega_\chi(f_1(x))$;
\item $f_1(\omega_\varphi(u))=\omega_\chi(f_1(u))$ and $f_2(\omega_\chi(u))=\omega_\varphi(f_2(u))$ for every $u\in\I$.
\end{enumerate}
\end{proposition}

\begin{remark}\label{rem:swap}
With $\sigma(x,y)=(y,x)$, the map $G=\sigma\circ F\circ\sigma$, $G(u,v)=(f_1(v),f_2(u))$, is a Cournot map whose
compositions are $f_1\circ f_2$ and $f_2\circ f_1$, in this order. We have $\omega_G(\sigma(z))=\sigma(\omega_F(z))$,
$\omega_{G^2}(\sigma(z))=\sigma(\omega_{F^2}(z))$ and $\omega_{G^2}(G(\sigma(z)))=\sigma(\omega_{F^2}(F(z)))$, and $\sigma$
exchanges the projections. So every statement about points with $\omega_{f_2\circ f_1}(x)$ finite has a mirror
statement for points with $\omega_{f_1\circ f_2}(y)$ finite, and conditions (C1), (C2) are invariant under $\sigma$.
\end{remark}

\section{The comb and the rigidity of its sections}\label{sec:comb}

Throughout this section $F(x,y)=(f_2(y),f_1(x))$, $z=(x,y)$, and $\omega_\varphi(x)$ is finite. Then $x$ is
asymptotically periodic and $\omega_\varphi(x)=P=\{p_0,\dots,p_{m-1}\}$ is a periodic orbit of period $m$;
we label it so that $\varphi(p_i)=p_{i+1}$ and $\varphi^{km+i}(x)\to p_i$ as $k\to\infty$ (indices modulo $m$).
We use the notation $M_i$, $q_i$, $N_i$, $Q$, $V_i$, $H_j$ of the Introduction.

\begin{lemma}\label{lem:relations}
For all $i$:
\begin{enumerate}[\upshape(a)]
\item $\chi(q_i)=q_{i+1}$ and $f_2(q_i)=p_{i+1}$;
\item $\chi(M_i)=M_{i+1}$, $\chi^m(M_i)=M_i$ and $\bigcup_iM_i=\omega_\chi(y)$;
\item $N_i=\omega_{\varphi^m}\bigl(f_2(\chi^i(y))\bigr)$, $\varphi(N_i)=N_{i+1}$ and $f_1(N_i)=M_{i+1}$.
\end{enumerate}
\end{lemma}

\begin{proof}
(a) $\chi(q_i)=f_1(f_2(f_1(p_i)))=f_1(\varphi(p_i))=f_1(p_{i+1})=q_{i+1}$ and $f_2(q_i)=\varphi(p_i)=p_{i+1}$.
(b) The first two identities follow from $\chi(\omega_{\chi^m}(u))=\omega_{\chi^m}(\chi(u))$ and strong invariance,
the third from \eqref{eq:split}.
(c) The first identity follows from $f_2(\omega_{\chi^m}(u))=\omega_{\varphi^m}(f_2(u))$, which is proved like
Proposition~\ref{prop:BLC}(c) using \eqref{eq:semiconj}; then
$\varphi(N_i)=f_2(\chi(M_i))=f_2(M_{i+1})=N_{i+1}$ and $f_1(N_i)=\chi(M_i)=M_{i+1}$.
\end{proof}

\begin{proposition}[Comb structure]\label{prop:comb}
$\omega_{F^2}(z)=\bigcup_iV_i$, $\omega_{F^2}(F(z))=\bigcup_jH_j$, and \eqref{eq:comb} holds. In particular
$\Int\omega_F(z)=\emptyset$ and $\pi_2(\omega_F(z))=\omega_\chi(y)\cup Q$.
\end{proposition}

\begin{proof}
We have $F^{2n}(z)=(\varphi^n(x),\chi^n(y))$. Let $(a,b)=\lim_kF^{2n_k}(z)$ with $n_k\to\infty$. Passing to a
subsequence we may assume $n_k\equiv i\pmod m$; then $a=p_i$ and $b\in\omega_{\chi^m}(\chi^i(y))=M_i$.
Conversely, if $b\in M_i$, then $b=\lim_k\chi^{km_k+i}(y)$ along some $m_k\to\infty$, and along the same
subsequence $\varphi^{km_k+i}(x)\to p_i$. Hence $\omega_{F^2}(z)=\bigcup_i\{p_i\}\times M_i$. By
Proposition~\ref{prop:BLC}(a),
\[
\omega_{F^2}(F(z))=F\Bigl(\bigcup_i\{p_i\}\times M_i\Bigr)=\bigcup_i f_2(M_i)\times\{f_1(p_i)\}=\bigcup_iH_i .
\]
The last two claims are immediate.
\end{proof}

\begin{figure}[t]
\centering
\begin{tikzpicture}[scale=5.2,>=Stealth]
  \draw[gray!60] (0,0) rectangle (1,1);
  \node[below left] at (0,0) {$0$}; \node[below] at (1,0) {$1$}; \node[left] at (0,1) {$1$};
  \draw[line width=2pt,orange!80!black] (-0.03,0.15)--(-0.03,0.85);
  \node[orange!70!black,left] at (-0.04,0.5) {$J$};
  \draw[line width=2pt,orange!80!black] (0.1,-0.03)--(0.9,-0.03);
  \node[orange!70!black,below] at (0.5,-0.04) {$K=f_2(J)$};
  \foreach \p/\l in {0.3/p_0,0.65/p_1}{
    \draw[line width=2.2pt,blue!70] (\p,0.15)--(\p,0.85);
    \node[blue!70!black,above] at (\p,0.86) {$\{\l\}\times J$};}
  \foreach \q/\l in {0.35/q_0,0.72/q_1}{
    \draw[line width=2.2pt,teal] (0.1,\q)--(0.9,\q);
    \node[teal!70!black,right] at (0.91,\q) {$K\times\{\l\}$};}
  \foreach \p in {0.3,0.65}{\foreach \q in {0.35,0.72}{\fill[red!70!black] (\p,\q) circle (0.012);}}
\end{tikzpicture}
\caption{A lattice comb (case (I) of Theorem~\ref{thm:B}) over a $2$-cycle $P=\{p_0,p_1\}$ of $\varphi=f_2\circ f_1$.
All vertical sections equal $J$, the heights $q_i=f_1(p_i)$ lie in $J$, and the abscissae $p_i$ lie in $f_2(J)$,
so every vertical piece meets every horizontal piece (red points).}
\label{fig:lattice}
\end{figure}
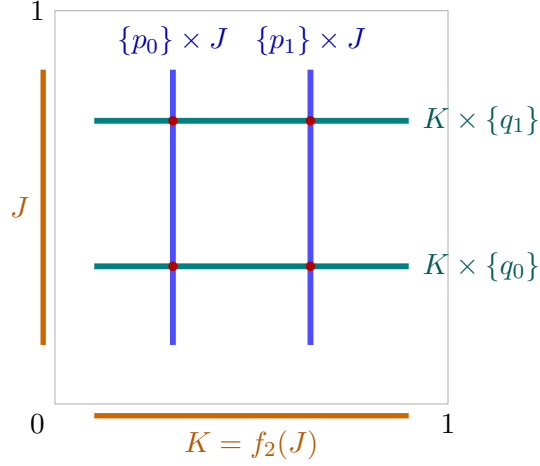

\begin{proposition}[Rigidity of the sections]\label{prop:rigidity}
Exactly one of the following holds:
\begin{enumerate}[\upshape(a)]
\item all $M_i$ are finite, with the same cardinality;
\item all $M_i$ are infinite and nowhere dense;
\item all $M_i$ are finite unions of nondegenerate closed intervals, with the same number of components.
\end{enumerate}
The same holds for the horizontal sections $N_0,\dots,N_{m-1}$.
\end{proposition}

\begin{proof}
Each $M_i$ is an $\omega$-limit set of the interval map $\chi^m$, so it is of one of the three types
\cite{ABCP,Sh}. Since $\chi(M_i)=M_{i+1}$, $\#M_{i+1}\le\#M_i$; going around the cycle, if one $M_i$ is finite
all are, with equal cardinalities.

Suppose that some $M_{i+1}$ has nonempty interior, so it is a finite union of nondegenerate intervals. Then
$M_i=\chi^{m-1}(M_{i+1})$ is a finite union of intervals and points. If $M_i$ were nowhere dense it would be
finite, and then so would $M_{i+1}=\chi(M_i)$, a contradiction. Hence $M_i$ has nonempty interior, and going
around the cycle all $M_i$ are of type (c). For type (c), $\chi$ maps each component of $M_i$ into a single
component of $M_{i+1}$, and $\chi(M_i)=M_{i+1}$ makes the induced map between components onto; so the number of
components cannot increase along the cycle and is therefore constant.

For the $N_i$ use $\varphi(N_i)=N_{i+1}$ and that $N_i$ is an $\omega$-limit set of $\varphi^m$ (Lemma~\ref{lem:relations}).
\end{proof}

\begin{remark}[Contrast with triangular maps]\label{rem:contrast}
For a triangular map $T(x,y)=(f(x),g(x,y))$ whose base orbit converges to an $m$-cycle, the $\omega$-limit set
is also contained in finitely many fibres, but the fibre maps depend on the base point. This non-autonomous
freedom allows sections of different types in a single $\omega$-limit set: for instance
$\{0\}\times\mathcal C\cup\{1\}\times(\frac13\mathcal C\cup[\frac23,1])$, with $\mathcal C$ the middle-thirds Cantor set,
is an $\omega$-limit set of a triangular map \cite[Thm.~D and Ex.~6.1]{Gu26}. By
Proposition~\ref{prop:rigidity}, no such mixing of types occurs among the vertical sections of a Cournot map,
because the second coordinate of $F^2$ evolves autonomously under $\chi$.
\end{remark}

\section{The propagation lemma}\label{sec:prop}

\begin{lemma}\label{lem:incidence}
$V_i\cap H_j\neq\emptyset$ if and only if $q_j\in M_i$ and $p_i\in N_j$; in that case $V_i\cap H_j=\{(p_i,q_j)\}$.
Moreover, $\omega_{F^2}(z)\cap\omega_{F^2}(F(z))\neq\emptyset$ if and only if $V_i\cap H_j\neq\emptyset$ for some $i,j$.
\end{lemma}

\begin{proof}
A point of $V_i$ has first coordinate $p_i$ and a point of $H_j$ has second coordinate $q_j$. The second
statement follows from Proposition~\ref{prop:comb}.
\end{proof}

We say that $(i,j)\in\mathbb{Z}_m^2$ is an \emph{incidence} if $V_i\cap H_j\neq\emptyset$.

\begin{lemma}[Propagation]\label{lem:propagation}
If $(i,j)$ is an incidence, then so are $(i+1,j+1)$ and $(j+1,i)$.
\end{lemma}

\begin{proof}
By Lemma~\ref{lem:incidence}, $q_j\in M_i$ and $p_i\in N_j$. By Lemma~\ref{lem:relations},
\[
q_{j+1}=\chi(q_j)\in\chi(M_i)=M_{i+1},\qquad p_{i+1}=\varphi(p_i)\in\varphi(N_j)=N_{j+1},
\]
so $(i+1,j+1)$ is an incidence. Also
\[
q_i=f_1(p_i)\in f_1(N_j)=M_{j+1},\qquad p_{j+1}=f_2(q_j)\in f_2(M_i)=N_i,
\]
so $(j+1,i)$ is an incidence.
\end{proof}

\begin{figure}[t]
\centering
\begin{tikzpicture}[scale=0.9,>=Stealth]
  \foreach \i in {0,...,3}{
    \node at (\i+0.5,4.35) {$H_{\i}$};
    \node at (-0.45,3.5-\i) {$V_{\i}$};}
  \draw[gray!50] (0,0) grid (4,4);
  \fill[red!60] (1,3) rectangle (2,4);
  \node[font=\small] at (1.5,3.5) {$(0,1)$};
  \fill[red!30] (2,2) rectangle (3,3); \node[font=\small] at (2.5,2.5) {$(1,2)$};
  \fill[red!30] (3,1) rectangle (4,2); \node[font=\small] at (3.5,1.5) {$(2,3)$};
  \fill[red!30] (0,0) rectangle (1,1); \node[font=\small] at (0.5,0.5) {$(3,0)$};
  \fill[blue!25] (0,1) rectangle (1,2); \node[font=\small] at (0.5,1.5) {$(2,0)$};
  \fill[blue!15] (1,0) rectangle (2,1); \node[font=\small] at (1.5,0.5) {$(3,1)$};
  \fill[blue!15] (2,3) rectangle (3,4); \node[font=\small] at (2.5,3.5) {$(0,2)$};
  \fill[blue!15] (3,2) rectangle (4,3); \node[font=\small] at (3.5,2.5) {$(1,3)$};
  \draw[thick] (0,0) rectangle (4,4);
\end{tikzpicture}
\caption{The propagation lemma for $m=4$. One incidence $(0,1)$ (dark red) generates, by the shift
$(i,j)\mapsto(i+1,j+1)$, a diagonal of incidences (light red), and by $(i,j)\mapsto(j+1,i)$ a second diagonal
(blue). In an alternating comb the incidences depend only on parities, and the two diagonals then cover all
cells, which is impossible (Proposition~\ref{prop:alternating}).}
\label{fig:propagation}
\end{figure}

\section{Proofs of the main results}\label{sec:proofs}

\begin{proposition}[Two shapes]\label{prop:shapes}
Assume $\pi_2(\omega_F(z))$ is connected. Then $\omega_\chi(y)=\bigcup_iM_i$ is a closed interval $J$, possibly
degenerate, and $Q\subset J$. If $J$ is nondegenerate, then either $M_i=J$ for all $i$, or $m$ is even and the
sections alternate between the two intervals $J_0,J_1$ of Theorem~\ref{thm:BM} for $\chi|_J$.
\end{proposition}

\begin{proof}
By Proposition~\ref{prop:comb}, $\pi_2(\omega_F(z))=U\cup Q$ with $U=\omega_\chi(y)$ and $Q$ finite.

If the $M_i$ are of type (a) or (b) of Proposition~\ref{prop:rigidity}, $U$ is closed and nowhere dense, hence so
is $U\cup Q$, and a connected nowhere dense subset of $\I$ is a point. So $U\cup Q=\{c\}$ and $J=\{c\}$.

Assume type (c). Then $U$ is a finite union of nondegenerate closed intervals. Finitely many points cannot fill a
gap of positive length, and a point of $Q$ outside $U$ would be an isolated point of $U\cup Q$. Hence $U$ is an
interval $J$ and $Q\subset J$. Since $J=\omega_\chi(y)$ has nonempty interior, some $w=\chi^{n_0}(y)$ lies in $J$
with $\omega_\chi(w)=J$, so $\chi|_J$ is transitive.

If $\chi|_J$ is mixing, then $\chi^m|_J$ is transitive, and each $M_i$ is a closed $\chi^m$-invariant subset of $J$ with
nonempty interior; by Lemma~\ref{lem:noinvariant}, $M_i=J$.

Otherwise let $J_0,J_1$ be as in Theorem~\ref{thm:BM}, so $\chi^{2m}|_{J_k}$ is transitive. The set $M_i$ has a
nondegenerate component, which meets the interior of some $J_k$; then $M_i\cap J_k$ is a closed
$\chi^{2m}$-invariant subset of $J_k$ with nonempty interior, so $J_k\subset M_i$ by Lemma~\ref{lem:noinvariant}.
Applying $\chi^m$ gives $J_{k+m}\subset M_i$ (index modulo $2$). If $m$ is odd, $M_i=J$. If $m$ is even,
the same argument applied to $J_{1-k}$ shows that either $M_i=J$ or $M_i=J_k$. In the first case
$M_{i'}=\chi^{i'-i}(J)=J$ for all $i'$; in the second, $M_{i+1}=\chi(J_k)=J_{1-k}$, and the sections alternate.
\end{proof}

\begin{proposition}[Lattice case]\label{prop:lattice}
If $\pi_2(\omega_F(z))$ is connected and $M_0=\dots=M_{m-1}=J$, then every $(i,j)$ is an incidence,
$\omega_F(z)=(P\times J)\cup(f_2(J)\times Q)$ is connected, and \textup{(C2)} holds.
\end{proposition}

\begin{proof}
By Proposition~\ref{prop:shapes}, $Q\subset J$, so $q_j\in M_i$ for all $i,j$. Also $N_j=f_2(J)$ and, by
Lemma~\ref{lem:relations}(a), $p_i=f_2(q_{i-1})\in f_2(J)$. By Lemma~\ref{lem:incidence} every $(i,j)$ is an
incidence. Each piece $V_i=\{p_i\}\times J$ and $H_j=f_2(J)\times\{q_j\}$ is connected, and every vertical
piece meets every horizontal piece, so the union is connected. Finally (C2) holds by Lemma~\ref{lem:incidence}.
\end{proof}

\begin{proposition}[Alternating case]\label{prop:alternating}
If $m$ is even and $M_i=J_{(i+k)\bmod2}$ as in Proposition~\ref{prop:shapes}, then there are no incidences.
Consequently \textup{(C2)} fails and $\omega_F(z)$ is not connected.
\end{proposition}

\begin{proof}
Here $M_{i+2}=M_i$, hence $N_{i+2}=N_i$ for all $i$. We first show that incidences depend only on parities.
If $(i,j)$ is an incidence, then $q_j\in M_i$ and $p_i\in N_j$, and
\[
q_{j+2}=\chi^2(q_j)\in\chi^2(M_i)=M_i,\qquad p_{i+2}=\varphi^2(p_i)\in\varphi^2(N_j)=N_j;
\]
together with $M_{i+2}=M_i$ and $N_{j+2}=N_j$ this shows that $(i+2,j)$ and $(i,j+2)$ are incidences.

Since $m$ is even, parities are well defined on $\mathbb{Z}_m$. Suppose that some $(i_0,j_0)$ is an
incidence, with parities $(a,b)$. Lemma~\ref{lem:propagation}, applied to $(i_0,j_0)$ and then to $(j_0+1,i_0)$,
gives incidences with parities $(a+1,b+1)$, $(b+1,a)$ and $(b,a+1)$. Whether $a=b$ or $a\neq b$, these four pairs
exhaust the four parity classes, so by parity invariance every $(i,j)\in\mathbb{Z}_m^2$ is an incidence. In
particular $(i,i)$ and $(i,i+1)$ are incidences for every $i$, that is, $q_i\in M_i$ and $q_{i+1}\in M_i$. Hence $q_i\in M_i\cap M_{i-1}=J_0\cap J_1=\{c\}$ for every $i$. By Lemma~\ref{lem:relations}(a),
$p_{i+1}=f_2(q_i)=f_2(c)$ for every $i$, so $P$ is a single point and $m=1$, contradicting that $m$ is even.

So there are no incidences, (C2) fails by Lemma~\ref{lem:incidence}, and $\omega_F(z)$ is not connected by
Theorem~\ref{thm:LM}.
\end{proof}

\begin{proof}[Proof of Theorem~\ref{thm:B}]
By Proposition~\ref{prop:shapes}, either all $M_i$ equal an interval $J$ (degenerate or not), or $m$ is even and
the sections alternate. The first case is Proposition~\ref{prop:lattice}, which also gives $P\subset f_2(J)$;
the second is Proposition~\ref{prop:alternating}. The two cases are exclusive, because in the second $M_i\neq M_{i+1}$.
\end{proof}

\begin{proof}[Proof of Theorem~\ref{thm:A}]
Necessity is Theorem~\ref{thm:LM}. For sufficiency assume (C1) and (C2).

If $\omega_{f_2\circ f_1}(x)$ is finite, $\pi_2(\omega_F(z))$ is connected by (C1), so Theorem~\ref{thm:B} applies.
Case (II) is excluded because (C2) holds, so we are in case (I) and $\omega_F(z)$ is connected.

If $\omega_{f_1\circ f_2}(y)$ is finite, let $\sigma(x,y)=(y,x)$. Then $G=\sigma\circ F\circ\sigma$,
$G(u,v)=(f_1(v),f_2(u))$, is a Cournot map whose compositions are $f_1\circ f_2$ and $f_2\circ f_1$ in this order. We
have $\omega_G(\sigma(z))=\sigma(\omega_F(z))$, $\omega_{G^2}(\sigma(z))=\sigma(\omega_{F^2}(z))$ and
$\omega_{G^2}(G(\sigma(z)))=\sigma(\omega_{F^2}(F(z)))$, and $\sigma$ exchanges the projections. So
(C1) and (C2) hold for $G$ and $\sigma(z)=(y,x)$, whose first coordinate has a finite $\omega$-limit set under
the first composition of $G$. By the first case $\omega_G(\sigma(z))$ is connected, hence so is $\omega_F(z)$.
\end{proof}

\begin{remark}
The proof shows that, when $\omega_{f_2\circ f_1}(x)$ is finite, (C2) and the connectedness of $\pi_2(\omega_F(z))$
alone imply connectedness; the connectedness of $\pi_1(\omega_F(z))$ then follows.
\end{remark}

\section{Piecewise linear maps with prescribed values}\label{sec:lemma}

A \emph{grid} of a nondegenerate compact interval $A$ is a finite subset $G\subset A$ containing both endpoints;
the closures of the components of $A\setminus G$ are the \emph{cells} of $G$. A continuous map $h\colon A\to B$ is
\emph{piecewise linear} if there is a finite set $\operatorname{Br}(h)\subset A$, containing the endpoints of $A$, such
that $h$ is affine and nonconstant on the closure of each component of $A\setminus\operatorname{Br}(h)$; these closures
are the \emph{laps} of $h$, and the points of $\operatorname{Br}(h)$ its \emph{breakpoints}. We always take
$\operatorname{Br}(h)$ minimal except that it may contain the points of a given grid.

\begin{definition}\label{def:good}
Let $A,B$ be nondegenerate compact intervals with grids $G_A,G_B$, and $\lambda>0$. A continuous map $h\colon A\to B$
is \emph{$\lambda$-good} (with respect to $G_A,G_B$) if it is piecewise linear and
\begin{enumerate}[\upshape(G1)]
\item $G_A\subset\operatorname{Br}(h)$ and $h(\operatorname{Br}(h))\subset G_B$;
\item $|h'|\ge\lambda$ on every lap;
\item $h(C)=B$ for every cell $C$ of $G_A$.
\end{enumerate}
\end{definition}

\begin{lemma}[Prescribed values]\label{lem:good}
Let $A,B$ be nondegenerate compact intervals with grids $G_A,G_B$, let $v\colon G_A\to G_B$ be any map and
$\lambda>0$. Then there is a $\lambda$-good map $h\colon A\to B$ with $h|_{G_A}=v$.
\end{lemma}

\begin{proof}
Write $B=[\gamma,\delta]$ and let $\eta>0$ be the minimal distance between two distinct points of $G_B$. Let
$C=[u,u']$ be a cell of $G_A$ and let $r\ge1$ be an integer. Consider the finite sequence
\[
v(u),\ \gamma,\ \delta,\ \gamma,\ \delta,\ \dots,\ \gamma,\ \delta,\ v(u')
\]
containing $r$ pairs $\gamma,\delta$, and delete every term that equals the preceding one. We obtain a sequence
$w_0=v(u),w_1,\dots,w_\ell=v(u')$ of points of $G_B$ with $w_{k}\neq w_{k+1}$ and $\ell\ge 2r-1$. Divide $C$ into
$\ell$ intervals of equal length by points $u=t_0<t_1<\dots<t_\ell=u'$ and let $h$ be affine on $[t_k,t_{k+1}]$ with
$h(t_k)=w_k$. On each of these intervals $|h'|=|w_{k+1}-w_k|\,\ell/(u'-u)\ge\eta(2r-1)/(u'-u)$, which is at least
$\lambda$ if $r$ is large. The values $\gamma$ and $\delta$ are attained on $C$, so $h(C)=B$. Doing this on every
cell, the definitions agree at the points of $G_A$, so $h$ is continuous, and (G1)--(G3) hold with
$\operatorname{Br}(h)=\{t_k\}\supset G_A$ for all cells.
\end{proof}

Figure~\ref{fig:good} shows such a map. The values at the grid points are arbitrary; the fast oscillations inside
each cell provide expansion and surjectivity.

\begin{figure}[t]
\centering
\begin{tikzpicture}[xscale=7.5,yscale=3.4]
  \draw[gray!60] (0,0) rectangle (1,1);
  \foreach \g in {0,0.3,0.55,1}{\draw[gray!50,dashed] (\g,0)--(\g,1); \fill (\g,-0.02) circle (0.006);}
  \foreach \g in {0,0.4,1}{\draw[gray!35,dotted] (0,\g)--(1,\g);}
  \node[left] at (0,0) {$\gamma$}; \node[left] at (0,1) {$\delta$}; \node[left] at (0,0.4) {$v(u)$};
  \node[below] at (0,-0.04) {$u$}; \node[below] at (0.3,-0.04) {$u_1$}; \node[below] at (0.55,-0.04) {$u_2$};
  \node[below] at (1,-0.04) {$u_3$};
  \draw[thick,blue!70!black] (0,0.4)--(0.075,0)--(0.15,1)--(0.225,0)--(0.3,1);
  \draw[thick,blue!70!black] (0.3,1)--(0.3625,0)--(0.425,1)--(0.4875,0)--(0.55,0.4);
  \draw[thick,blue!70!black] (0.55,0.4)--(0.614,0)--(0.678,1)--(0.742,0)--(0.806,1)--(0.871,0)--(0.935,1)--(1,0);
  \foreach \x/\y in {0/0.4,0.3/1,0.55/0.4,1/0}{\fill[red!70!black] (\x,\y) circle (0.008);}
\end{tikzpicture}
\caption{A $\lambda$-good map $h\colon A\to B$ (Lemma~\ref{lem:good}) for the grid $G_A=\{u,u_1,u_2,u_3\}$ and
$G_B=\{\gamma,v(u),\delta\}$. The prescribed values at the grid points (red) are kept, every lap joins two points of
$G_B$ with slope at least $\lambda$, and every cell is mapped onto $B$. (In the first cell the repeated value
$\delta$ at the end of the sequence $v(u),\gamma,\delta,\gamma,\delta,v(u_1)$ has been deleted, as in the proof.)}
\label{fig:good}
\end{figure}
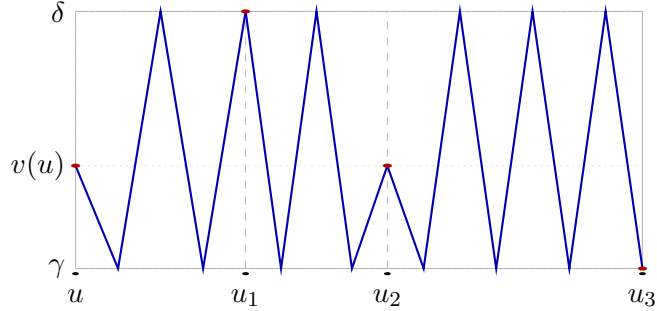

\begin{proposition}[Exact compositions]\label{prop:exact}
Let $r\ge1$, let $A_0,A_1,\dots,A_r=A_0$ be nondegenerate compact intervals with grids $G_0,G_1,\dots,G_r=G_0$,
and let $h_k\colon A_{k-1}\to A_k$ be $\lambda$-good with respect to $G_{k-1},G_k$, for $k=1,\dots,r$, where
$\lambda^r>2$. Then $H=h_r\circ\dots\circ h_1\colon A_0\to A_0$ is topologically exact.
\end{proposition}

\begin{proof}
Put $H_k=h_k\circ\dots\circ h_1$ ($H_0=\mathrm{id}$). By (G1), $h_k(G_{k-1})\subset G_k$ for every $k$, hence
\begin{equation}\label{eq:gridpres}
h_r\circ\dots\circ h_{k+1}(G_k)\subset G_r=G_0\qquad(0\le k\le r).
\end{equation}

\emph{Step 1: the breakpoints of $H$ are mapped into $G_0$.} Let $E\subset A_0$ be the set of points $e$ such that
$H_{k-1}(e)\in\operatorname{Br}(h_k)$ for some $k$. Off $E$ every $h_k$ is affine near $H_{k-1}(e)$, so $H$ is affine and
nonconstant (its slope is a product of nonzero slopes) on the closure of each component of $A_0\setminus E$; moreover
$E$ is finite and contains the endpoints of $A_0$. If $e\in E$ and $H_{k-1}(e)\in\operatorname{Br}(h_k)$, then
$H_k(e)\in G_k$ by (G1) and $H(e)\in G_0$ by \eqref{eq:gridpres}. On each component of $A_0\setminus E$ we have
$|H'|\ge\lambda^r$ by (G2).

\emph{Step 2: a full piece covers a cell.} Let $e<e'$ be consecutive points of $E$. Then $H$ is strictly monotone on
$[e,e']$ and $H(e)\neq H(e')$ are points of $G_0$, so $H([e,e'])$ contains a cell $C$ of $G_0$. By (G3) and
surjectivity, $h_1(C)=A_1$, $h_2(A_1)=A_2,\dots$, so $H(C)=A_0$. Hence $H^2([e,e'])=A_0$.

\emph{Step 3: expansion.} Let $U\subset A_0$ be a nondegenerate interval which contains no interval $[e,e']$ as in
Step~2. Then $U$ contains at most one point of $E$ in its interior, so $U=U_1\cup U_2$ with $H$ affine on each $U_j$, and
\[
|H(U)|\ \ge\ \max_j|H(U_j)|\ \ge\ \lambda^r\max_j|U_j|\ \ge\ \tfrac{\lambda^r}{2}\,|U|,
\]
where $|\cdot|$ denotes length. Since $\lambda^r/2>1$ and the lengths of the intervals $H^n(U)$ are bounded by $|A_0|$,
some $H^n(U)$ contains an interval $[e,e']$ as in Step~2, and then $H^{n+2}(U)=A_0$. Every nonempty open subset of
$A_0$ contains a nondegenerate interval, so $H$ is exact.
\end{proof}

\begin{remark}\label{rem:rotation}
The hypotheses of Proposition~\ref{prop:exact} are invariant under cyclic rotation of the chain $(h_1,\dots,h_r)$.
Hence all the compositions $h_k\circ\dots\circ h_1\circ h_r\circ\dots\circ h_{k+1}\colon A_k\to A_k$ are exact as well.
\end{remark}

\section{Realization of lattices}\label{sec:lattices}

\begin{figure}[t]
\centering
\begin{tikzpicture}[scale=5.0]
  \draw[gray!60] (0,0) rectangle (1,1);
  \node[below left] at (0,0) {$0$}; \node[below] at (1,0) {$1$}; \node[left] at (0,1) {$1$};
  \def\Ja{0.12} \def\Jb{0.88} \def\Ka{0.18} \def\Kb{0.95}
  \draw[line width=2pt,orange!80!black] (-0.03,\Ja)--(-0.03,\Jb); \node[orange!70!black,left] at (-0.04,0.5) {$J$};
  \draw[line width=2pt,orange!80!black] (\Ka,-0.03)--(\Kb,-0.03); \node[orange!70!black,below] at (0.56,-0.04) {$K$};
  \foreach \p/\l in {0.27/p_0,0.52/p_1,0.86/p_2}{
    \draw[line width=2.2pt,blue!70] (\p,\Ja)--(\p,\Jb); \node[blue!70!black,above] at (\p,\Jb) {$\l$};}
  \foreach \q/\l in {0.22/q_0,0.47/q_1,0.71/q_2}{
    \draw[line width=2.2pt,teal] (\Ka,\q)--(\Kb,\q); \node[teal!70!black,right] at (\Kb,\q) {$\l$};}
  \foreach \p in {0.27,0.52,0.86}{\foreach \q in {0.22,0.47,0.71}{\fill[red!70!black] (\p,\q) circle (0.011);}}
\end{tikzpicture}
\caption{A lattice of order $3$, $(P\times J)\cup(K\times Q)$ with $P\subset K$ and $Q\subset J$. By
Theorem~\ref{thm:C}(1) it is the $\omega$-limit set of $(p_0,y)$ for a Cournot map with $f_1(p_i)=q_i$ and
$f_2(q_i)=p_{i+1}$, where $y$ has a dense $\chi^3$-orbit in $J$.}
\label{fig:lattice3}
\end{figure}
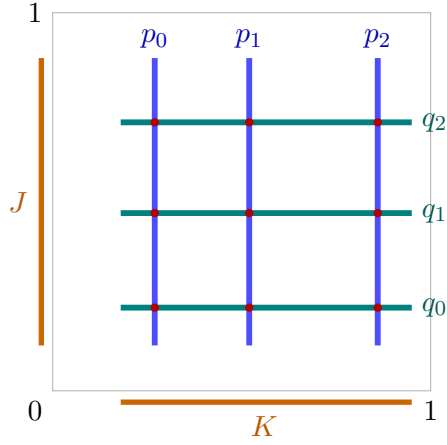

\begin{proof}[Proof of Theorem~\ref{thm:C}(1), necessity]
Let $L=\omega_F(z)$ be connected with more than one point, $z=(x,y)$ and $\omega_\varphi(x)=P$ finite, of cardinality
$m$. Then $\pi_2(L)$ is connected and, by Theorem~\ref{thm:B}, we are in case (I):
$L=(P\times J)\cup(K\times Q)$ with $J=\omega_\chi(y)$ an interval, $K=f_2(J)$, $Q\subset J$ and $P\subset K$.

The points $q_i$ are pairwise distinct: $f_2(q_i)=p_{i+1}$ by Lemma~\ref{lem:relations}(a), and the $p_i$ are
distinct. Hence $\#Q=\#P=m$.

If $J$ were degenerate, $J=\{c\}$, then $Q\subset J$ would force $m=1$ and $K=\{f_2(c)\}$, so $L$ would be a single
point. Hence $J$ is nondegenerate. Finally $K$ is nondegenerate, since $f_1(K)=f_1(N_0)=M_1=J$ by
Lemma~\ref{lem:relations}(c). So $L$ is a lattice.
\end{proof}

\begin{proof}[Proof of Theorem~\ref{thm:C}(1), sufficiency]
Let $\Lat=(P\times J)\cup(K\times Q)$ be a lattice of order $m$ and enumerate $P=\{p_0,\dots,p_{m-1}\}$ and
$Q=\{q_0,\dots,q_{m-1}\}$ arbitrarily. Let $G_K$ be the grid of $K$ formed by $P$ and the endpoints of $K$, and $G_J$
the grid of $J$ formed by $Q$ and the endpoints of $J$. Define $v_1\colon G_K\to G_J$ and $v_2\colon G_J\to G_K$ by
\[
v_1(p_i)=q_i,\qquad v_2(q_i)=p_{i+1}\quad(\text{indices modulo } m),
\]
and send the remaining grid points (endpoints of $K$ or $J$ not in $P$ or $Q$) to an endpoint of $J$ or $K$
respectively. These maps are well defined because the $p_i$ are distinct and the $q_i$ are distinct. By
Lemma~\ref{lem:good} with $\lambda=2$ there are $2$-good maps $g_1\colon K\to J$ and $g_2\colon J\to K$ with
$g_1|_{G_K}=v_1$ and $g_2|_{G_J}=v_2$. Extend them to continuous maps $f_1,f_2\colon\I\to\I$ which are constant on each
component of $\I\setminus K$ and of $\I\setminus J$, respectively.

By (G3), $f_1(K)=J$ and $f_2(J)=K$. By Proposition~\ref{prop:exact} applied to the chain $K\xrightarrow{g_1}J
\xrightarrow{g_2}K$ ($\lambda^2=4>2$) and Remark~\ref{rem:rotation}, the maps $\varphi|_K=g_2\circ g_1$ and
$\chi|_J=g_1\circ g_2$ are exact. Moreover
\[
\varphi(p_i)=f_2(f_1(p_i))=f_2(q_i)=p_{i+1},
\]
so $P$ is a periodic orbit of $\varphi$ of period $m$, and $\chi(q_i)=f_1(p_{i+1})=q_{i+1}$.

The map $\chi^m|_J$ is exact, hence transitive, so there is $y\in J$ with $\omega_{\chi^m}(y)=J$. Let $x=p_0$ and
$z=(x,y)$. Then $\omega_\varphi(x)=P$, with the labelling of Section~\ref{sec:prelim}, and
\[
M_i=\omega_{\chi^m}(\chi^i(y))=\chi^i\bigl(\omega_{\chi^m}(y)\bigr)=\chi^i(J)=J,\qquad N_i=f_2(M_i)=K .
\]
By \eqref{eq:comb}, $\omega_F(z)=(P\times J)\cup(K\times Q)=\Lat$, which is connected.
\end{proof}

\begin{proof}[Proof of Corollary~\ref{cor:D}]
By Theorem~\ref{thm:C}(1) and Remark~\ref{rem:swap}, $\omega_F(z)$ is a lattice $(P\times J)\cup(K\times Q)$ of some order
$m$. A vertical line $\{t\}\times\I$ meets $K\times Q$ in at most $m$ points, so the nondegenerate vertical segments of
the lattice lie in $P\times J$; and since $Q\subset J$, the line $\{p\}\times\I$, $p\in P$, meets the lattice exactly in
$\{p\}\times J$. Hence the maximal vertical segments are the $m$ segments $\{p\}\times J$, and likewise the maximal
horizontal segments are the $m$ segments $K\times\{q\}$. Every vertical one meets every horizontal one. The letter H has two
maximal vertical segments and one maximal horizontal segment; the letters E and F have one vertical segment and three,
respectively two, horizontal ones, and in the letter F the horizontal segments do not all have the same projection
either.
\end{proof}

\begin{example}\label{ex:LM}
The lattices $(\I\times P')\cup(P'\times\I)$ of \cite[Example~1]{LM}, obtained with $f_1$ the tent map and $f_2=\mathrm{id}$,
are the images under $\sigma$ of lattices with $J=K=\I$ and $P=Q=P'$. Theorem~\ref{thm:C} shows that the much
larger family of all lattices, with arbitrary intervals $J,K$ and arbitrary finite sets $P\subset K$, $Q\subset J$ of
equal cardinality, is realized.
\end{example}

\begin{remark}\label{rem:unimodal}
The reaction functions produced by the proof have many turning points. In the economic interpretation one usually
requires unimodal reaction functions, for instance logistic ones \cite{Pu}. Whether every lattice is realized by
a Cournot map with unimodal $f_1,f_2$ is an open question (Problem~\ref{prob:unimodal}).
\end{remark}

\section{Realization of alternating combs}\label{sec:alternating}

\begin{lemma}\label{lem:disjoint}
The $2m$ segments of an alternating comb are pairwise disjoint.
\end{lemma}

\begin{proof}
Two vertical segments lie on different vertical lines since the $p_i$ are distinct, and similarly for horizontal
ones. Suppose that $\{p_i\}\times J_i$ meets $K_j\times\{q_j\}$. Then $q_j\in J_i$ and $p_i\in K_j$. Since
$q_j\in J_{j+s}\setminus\{c\}$ and $J_0\cap J_1=\{c\}$, we get $i\equiv j+s\pmod2$; since $p_i\in K_{i+s+1}\setminus\{d\}$,
we get $j\equiv i+s+1\pmod2$. Adding, $0\equiv 2s+1\equiv1\pmod 2$, a contradiction.
\end{proof}

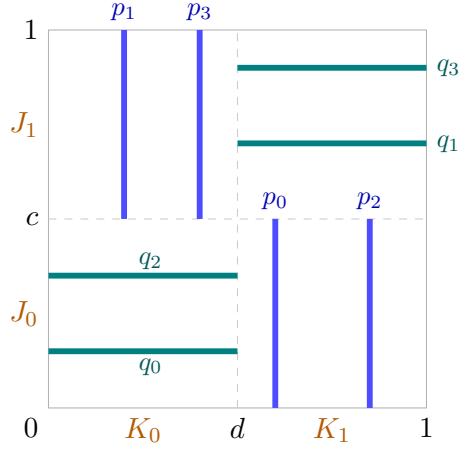
\begin{figure}[t]
\centering
\begin{tikzpicture}[scale=5.0]
  \draw[gray!60] (0,0) rectangle (1,1);
  \node[below left] at (0,0) {$0$}; \node[below] at (1,0) {$1$}; \node[left] at (0,1) {$1$};
  \draw[gray!40,dashed] (0.5,0)--(0.5,1); \draw[gray!40,dashed] (0,0.5)--(1,0.5);
  \node[left] at (0,0.5) {$c$}; \node[below] at (0.5,0) {$d$};
  \node[left,orange!70!black] at (0,0.25) {$J_0$}; \node[left,orange!70!black] at (0,0.75) {$J_1$};
  \node[below,orange!70!black] at (0.25,0) {$K_0$}; \node[below,orange!70!black] at (0.75,0) {$K_1$};
  \foreach \p/\a/\b/\l in {0.6/0/0.5/p_0,0.2/0.5/1/p_1,0.85/0/0.5/p_2,0.4/0.5/1/p_3}{
    \draw[line width=2.2pt,blue!70] (\p,\a)--(\p,\b);}
  \node[blue!70!black,above] at (0.6,0.5) {\small$p_0$}; \node[blue!70!black,above] at (0.2,1) {\small$p_1$};
  \node[blue!70!black,above] at (0.85,0.5) {\small$p_2$}; \node[blue!70!black,above] at (0.4,1) {\small$p_3$};
  \foreach \q/\a/\b in {0.15/0/0.5,0.7/0.5/1,0.35/0/0.5,0.9/0.5/1}{
    \draw[line width=2.2pt,teal] (\a,\q)--(\b,\q);}
  \node[teal!70!black,right] at (1,0.7) {\small$q_1$}; \node[teal!70!black,right] at (1,0.9) {\small$q_3$};
  \node[teal!70!black] at (0.27,0.11) {\small$q_0$}; \node[teal!70!black] at (0.27,0.39) {\small$q_2$};
\end{tikzpicture}
\caption{An alternating comb of order $4$ with $s=0$: $J_0=K_0=[0,\frac12]$, $J_1=K_1=[\frac12,1]$, vertical segments
$\{p_i\}\times J_i$ and horizontal segments $K_i\times\{q_i\}$, with $q_i\in J_i$ and $p_i\in K_{i+1}$. The eight
segments are pairwise disjoint, while both projections are $[0,1]$. By Theorem~\ref{thm:C}(2) this set is an
$\omega$-limit set of a Cournot map.}
\label{fig:altcomb}
\end{figure}

\begin{proof}[Proof of Theorem~\ref{thm:C}(2), necessity]
Let $L=\omega_F(z)$ be as in Theorem~\ref{thm:C}(2). By Theorem~\ref{thm:B} and Proposition~\ref{prop:shapes} we are in case (II): $m$ is even, $J=\omega_\chi(y)$
is a nondegenerate interval, $\chi|_J$ is transitive but not mixing, $J=J_0\cup J_1$ with $J_0\cap J_1=\{c\}$,
$\chi(J_0)=J_1$, $\chi(J_1)=J_0$, and $M_i=J_{(i+k)\bmod2}$. Exchanging the names of $J_0$ and $J_1$ if necessary we
may assume $k=0$, so $M_i=J_i$ (indices of $J$ modulo $2$). Put $K_a=f_2(J_a)$ and $K=f_2(J)=K_0\cup K_1$, so that
$N_i=K_i$ and, by \eqref{eq:comb},
\[
L=\bigcup_i\bigl(\{p_i\}\times J_i\bigr)\cup\bigcup_i\bigl(K_i\times\{q_i\}\bigr).
\]

\emph{The intervals $K_0,K_1$.} They are compact intervals, nondegenerate because $f_1(K_a)=f_1(N_a)=M_{a+1}=J_{a+1}$.
By Proposition~\ref{prop:BLC}(c), $K=f_2(\omega_\chi(y))=\omega_\varphi(f_2(y))$, so $K$ is an
interval which is an $\omega$-limit set of $\varphi$ with nonempty interior; hence some point of the $\varphi$-orbit of
$f_2(y)$ lies in $K$ and has $K$ as its $\omega$-limit set, and $\varphi|_K$ is transitive. If $\varphi|_K$ were mixing,
$\varphi^m|_K$ would be transitive and, by Lemma~\ref{lem:noinvariant} applied to the closed $\varphi^m$-invariant sets
$N_0=K_0$ and $N_1=K_1$, we would get $K_0=K_1=K$ and then $J_1=f_1(K_0)=f_1(K_1)=J_0$, which is absurd. So $\varphi|_K$ is
not mixing, and Theorem~\ref{thm:BM} gives a splitting $K=K_0'\cup K_1'$ with $K_0'\cap K_1'=\{d\}$ exchanged by $\varphi$.
Exactly as in the proof of Proposition~\ref{prop:shapes}, the closed $\varphi^{2m}$-invariant set $N_0$ with nonempty interior
is $K$, $K_0'$ or $K_1'$; it is not $K$, since then $N_1=\varphi(N_0)=K$ as well and we are back in the previous
contradiction. Hence $\{K_0,K_1\}=\{K_0',K_1'\}$ and $K_0\cap K_1=\{d\}$, with $\varphi(d)=d$.

\emph{The points $q_i$.} By Proposition~\ref{prop:shapes} and Lemma~\ref{lem:relations}, $Q\subset J$ is a periodic orbit of $\chi$, and its points are
distinct because $f_2(q_i)=p_{i+1}$; so its period is $m\ge2$, and $c\notin Q$ since $\chi(c)=c$. As $\chi(J_a)=J_{a+1}$
and $\chi(q_i)=q_{i+1}\neq c$, the points $q_i$ alternate between $J_0\setminus\{c\}$ and $J_1\setminus\{c\}$: there is
$s\in\{0,1\}$ with $q_i\in J_{i+s}\setminus\{c\}$ for all $i$ (here we use that $m$ is even).

\emph{The points $p_i$.} We have $p_{i+1}=f_2(q_i)\in f_2(J_{i+s})=K_{i+s}=K_{(i+1)+s+1}$, and $p_{i+1}\neq d$ because
$P$ has period $m\ge2$ and $\varphi(d)=d$. Thus $p_i\in K_{i+s+1}\setminus\{d\}$ for all $i$, and $L$ is an alternating
comb.
\end{proof}

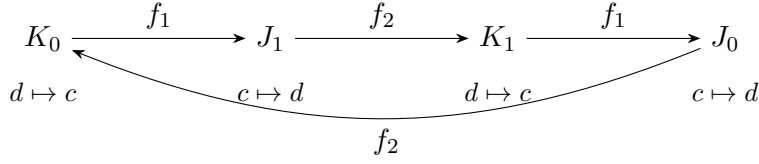
\begin{figure}[t]
\centering
\begin{tikzpicture}[>=Stealth,node distance=2.3cm]
  \node (K0) {$K_0$};
  \node[right=of K0] (J1) {$J_1$};
  \node[right=of J1] (K1) {$K_1$};
  \node[right=of K1] (J0) {$J_0$};
  \draw[->] (K0)--node[above]{$f_1$} (J1);
  \draw[->] (J1)--node[above]{$f_2$} (K1);
  \draw[->] (K1)--node[above]{$f_1$} (J0);
  \draw[->] (J0) to[bend left=22] node[below]{$f_2$} (K0);
  \node[below=0.15cm of K0,font=\small] {$d\mapsto c$};
  \node[below=0.15cm of J1,font=\small] {$c\mapsto d$};
  \node[below=0.15cm of K1,font=\small] {$d\mapsto c$};
  \node[below=0.15cm of J0,font=\small] {$c\mapsto d$};
\end{tikzpicture}
\caption{The chain of $2$-good maps used to realize an alternating comb. Each arrow is a piecewise linear surjection
with prescribed values on the grid ($p_i\mapsto q_i$ for $f_1$, $q_i\mapsto p_{i+1}$ for $f_2$, and the common endpoints
$c,d$ sent to each other). Going once around the cycle gives $\varphi^2|_{K_0}$, and starting at another vertex gives
$\chi^2|_{J_1}$, $\varphi^2|_{K_1}$, $\chi^2|_{J_0}$; all four are exact by Proposition~\ref{prop:exact}.}
\label{fig:chain}
\end{figure}

\begin{proof}[Proof of Theorem~\ref{thm:C}(2), sufficiency]
Let $\Gamma$ be an alternating comb, with the notation of Definition~\ref{def:alternating}. For
$a\in\{0,1\}$ let $G_{K_a}$ be the grid of $K_a$ formed by its endpoints (among them $d$) and the points $p_i$ lying in
$K_a$, and $G_{J_a}$ the grid of $J_a$ formed by its endpoints (among them $c$) and the points $q_i$ lying in $J_a$.

We define $v_1^a\colon G_{K_a}\to G_{J_{a+1}}$ and $v_2^a\colon G_{J_a}\to G_{K_a}$ by
\[
v_1^a(p_i)=q_i,\quad v_1^a(d)=c,\qquad v_2^a(q_i)=p_{i+1},\quad v_2^a(c)=d,
\]
and by sending the other endpoint of $K_a$ (resp.\ $J_a$), if it is not one of the $p_i$ (resp.\ $q_i$), to $c$
(resp.\ $d$). These assignments are consistent with
the targets: if $p_i\in K_a$, then $a\equiv i+s+1$ and $q_i\in J_{i+s}=J_{a+1}$; if $q_i\in J_a$, then $a\equiv i+s$ and
$p_{i+1}\in K_{i+s}=K_a$. They are well defined because the $p_i$ are distinct, the $q_i$ are distinct, $d\notin P$ and
$c\notin Q$.

By Lemma~\ref{lem:good} there are $2$-good maps $g_1^a\colon K_a\to J_{a+1}$ and $g_2^a\colon J_a\to K_a$ extending
$v_1^a$ and $v_2^a$. Since $g_1^0(d)=g_1^1(d)=c$, the maps $g_1^0,g_1^1$ define a continuous map $g_1\colon K\to J$;
likewise $g_2^0,g_2^1$ define $g_2\colon J\to K$. Extend $g_1,g_2$ to $f_1,f_2\colon\I\to\I$, constant on the
components of $\I\setminus K$ and $\I\setminus J$.

By (G3), $f_1(K_a)=J_{a+1}$ and $f_2(J_a)=K_a$, so $\chi(J_a)=J_{a+1}$ and $\varphi(K_a)=K_{a+1}$. The chain of
Figure~\ref{fig:chain},
\[
K_0\xrightarrow{\ g_1^0\ }J_1\xrightarrow{\ g_2^1\ }K_1\xrightarrow{\ g_1^1\ }J_0\xrightarrow{\ g_2^0\ }K_0,
\]
satisfies the hypotheses of Proposition~\ref{prop:exact} with $r=4$ and $\lambda=2$; by Remark~\ref{rem:rotation},
$\chi^2|_{J_0}=g_1^1\circ g_2^1\circ g_1^0\circ g_2^0$ is exact. As $m$ is even, $\chi^m|_{J_0}=(\chi^2|_{J_0})^{m/2}$ is exact,
so there is $y\in J_0$ with $\omega_{\chi^m}(y)=J_0$. Put $x=p_0$ and $z=(x,y)$.

As in the lattice case, $\varphi(p_i)=f_2(q_i)=p_{i+1}$, so $\omega_\varphi(x)=P$ with the correct labelling, and
$f_1(p_i)=q_i$. Furthermore
\[
M_i=\chi^i\bigl(\omega_{\chi^m}(y)\bigr)=\chi^i(J_0)=J_i,\qquad N_i=f_2(J_i)=K_i .
\]
By \eqref{eq:comb}, $\omega_F(z)=\Gamma$. Its vertical projection $J\cup Q=J$ is connected, and $\Gamma$ is not connected
by Lemma~\ref{lem:disjoint}.
\end{proof}

\begin{example}\label{ex:LM36}
For $m=2$, $s=0$, $J_0=K_0=[0,\frac12]$, $J_1=K_1=[\frac12,1]$ the alternating combs are the sets
\[
\bigl(\{p_0\}\times[0,\tfrac12]\bigr)\cup\bigl(\{p_1\}\times[\tfrac12,1]\bigr)\cup\bigl([0,\tfrac12]\times\{q_0\}\bigr)
\cup\bigl([\tfrac12,1]\times\{q_1\}\bigr)
\]
with $p_0\in(\frac12,1]$, $p_1\in[0,\frac12)$, $q_0\in[0,\frac12)$, $q_1\in(\frac12,1]$. Up to the swap $\sigma$, the
disconnected $\omega$-limit set of \cite[Fig.~36]{LM} is of this type. Theorem~\ref{thm:B} shows that all
choices of the four points occur.
\end{example}

\begin{table}[t]
\centering
\renewcommand{\arraystretch}{1.25}
\begin{tabular}{@{}p{3.3cm}p{5.6cm}p{2.6cm}@{}}
\toprule
Shape of $\omega_F(z)$ & Data & Realized by \\
\midrule
point & --- & constant maps \\
lattice of order $m\ge1$ & $J,K$; $P\subset K$, $Q\subset J$, $\#P=\#Q=m$ & Theorem~\ref{thm:C}(1) \\
alternating comb, $m$ even & $J_0\cup J_1$, $K_0\cup K_1$, $s$; $p_i,q_i$ as in Def.~\ref{def:alternating} & Theorem~\ref{thm:C}(2)\\
\bottomrule
\end{tabular}
\caption{All $\omega$-limit sets $\omega_F(z)$ of Cournot maps at points $z=(x,y)$ with $\omega_{f_2\circ f_1}(x)$ finite
and $\pi_2(\omega_F(z))$ connected (Theorems~\ref{thm:B} and~\ref{thm:C}).}
\label{tab:summary}
\end{table}

\section{Almost every point}\label{sec:ae}

\begin{definition}\label{def:nonsingular}
A continuous map $f\colon\I\to\I$ is \emph{nonsingular} if $\Leb(f(N))=0$ and $\Leb(f^{-1}(N))=0$ for every Lebesgue null
set $N\subset\I$.
\end{definition}

$C^1$ maps with finitely many critical points and piecewise linear maps without intervals of constancy are
nonsingular. Compositions of nonsingular maps are nonsingular. If $f$ is nonsingular and $\mu$ is a probability
absolutely continuous with respect to $\Leb$, then so is $f_*\mu$.

\begin{definition}\label{def:regular}
An interval map $\varphi$ is \emph{metrically regular} if there are finitely many closed sets $A_1,\dots,A_k\subset\I$
with $\varphi(A_j)=A_j$ such that $\omega_\varphi(x)\in\{A_1,\dots,A_k\}$ for Lebesgue almost every $x\in\I$, and each $A_j$
is of one of the following types:
\begin{enumerate}
\item[\upshape(P)] a periodic orbit;
\item[\upshape(S)] a \emph{solenoid}: a Cantor set such that for every $\eps>0$ there are a nondegenerate compact interval
$W$ and $p\ge1$ with $\varphi^p(W)\subset W$, $A_j\subset\bigcup_{i<p}\varphi^i(W)$ and $|\varphi^i(W)|<\eps$ for $0\le i<p$;
\item[\upshape(I)] a cycle of intervals carrying a $\varphi$-invariant probability $\mu_j$, equivalent to Lebesgue measure on
$A_j$, such that for some $\ell\ge1$
\[
\mu_j=\tfrac1\ell\bigl(\mu_{j,1}+\dots+\mu_{j,\ell}\bigr),
\]
where each $\mu_{j,t}$ is a $\varphi^\ell$-invariant probability, exact for $\varphi^\ell$, whose support has nonempty
interior.
\end{enumerate}
\end{definition}

\begin{remark}\label{rem:regular}
(1) For $C^3$ interval maps with non-flat critical points, Lebesgue almost every point has as
$\omega$-limit set one of finitely many attractors, which are periodic orbits, cycles of intervals or Cantor sets
\cite{BL91,BPP,vSV}; the Cantor attractors are solenoids or \emph{wild} attractors \cite{BKNS}. So
Definition~\ref{def:regular} excludes wild attractors and cycles of intervals without an absolutely continuous
invariant probability of the required kind.

(2) The exactness requirement in (I) holds with $\ell=1$ whenever $\varphi|_{A_j}$ is exact with respect to an invariant
probability equivalent to Lebesgue measure, as in the examples below. More generally, Ledrappier \cite{Le81} proved
that for piecewise $C^2$ maps with non-flat critical points an ergodic absolutely continuous invariant probability
with positive entropy, all of whose powers are ergodic, is isomorphic to a Bernoulli shift.

(3) Hyperbolic maps, for which every critical point lies in the basin of an attracting periodic orbit, are metrically
regular with all attractors of type (P), and they are dense in the $C^r$ topology \cite{KSvS}. The map
$\varphi=g\circ g$, $g(t)=4t(1-t)$, is metrically regular with a single attractor $[0,1]$ of type (I): $g$ is conjugate
to the tent map by $t=\sin^2(\pi s/2)$, so the measure with density $1/(\pi\sqrt{t(1-t)})$ is $g$-invariant, exact,
and equivalent to Lebesgue measure.
\end{remark}

From now on $F(x,y)=(f_2(y),f_1(x))$ has nonsingular reaction functions and $\varphi=f_2\circ f_1$ is metrically
regular, with attractors $A_1,\dots,A_k$. Let $X_0$ be the full-measure set of $x$ with $\omega_\varphi(x)\in\{A_j\}$.

\begin{lemma}\label{lem:typical}
\begin{enumerate}[\upshape(a)]
\item $\omega_\chi(y)=f_1\bigl(\omega_\varphi(f_2(y))\bigr)$ for every $y\in\I$.
\item The set $Z_0=\{(x,y):x\in X_0,\ f_2(y)\in X_0\}$ has full Lebesgue measure in $\I^2$.
\item If $A_j$ is of type \textup{(I)}, then $B_j=f_1(A_j)$ is a cycle of intervals of $\chi$, and $\nu_j=(f_1)_*\mu_j$ is a
$\chi$-invariant probability equivalent to Lebesgue measure on $B_j$, which is the average of the $\chi^\ell$-invariant
probabilities $\nu_{j,t}=(f_1)_*\mu_{j,t}$; each $\nu_{j,t}$ is exact for $\chi^\ell$ and its support has nonempty interior.
\end{enumerate}
\end{lemma}

\begin{proof}
(a) $\chi^{n+1}(y)=f_1(\varphi^n(f_2(y)))$, and $f_1$ is continuous. (b) $\I\setminus X_0$ is null and $f_2$ is
nonsingular, so $\{y:f_2(y)\notin X_0\}$ is null; apply Fubini.

(c) By \eqref{eq:semiconj}, $f_1$ is a factor map from $(\varphi,\mu_j)$ to $(\chi,\nu_j)$ and from $(\varphi^\ell,\mu_{j,t})$ to
$(\chi^\ell,\nu_{j,t})$; factors of invariant, respectively exact, systems are invariant, respectively exact. If $\omega_\varphi(x)=A_j$, then
$B_j=\omega_\chi(f_1(x))$ by Proposition~\ref{prop:BLC}(c); as $f_1$ is nonsingular, it is not constant on any interval,
so $B_j$ and $\operatorname{supp}\nu_{j,t}=f_1(\operatorname{supp}\mu_{j,t})$ have nonempty interior, and $B_j$ is a
cycle of intervals of $\chi$. Since $f_1$ is nonsingular, $\nu_j\ll\Leb$. Conversely, if
$E\subset B_j$ has positive Lebesgue measure, then $f_1^{-1}(E)\cap A_j$ is not null, for otherwise its image, which
contains $E$, would be null; as $\mu_j$ is equivalent to Lebesgue measure on $A_j$, $\nu_j(E)=\mu_j(f_1^{-1}(E))>0$.
\end{proof}

\begin{lemma}[Solenoids are separated]\label{lem:solenoid}
Let $A$ be a solenoid of $\varphi$.
\begin{enumerate}[\upshape(a)]
\item If $K$ is a transitive cycle of intervals of $\varphi$, then $A\cap\Int K=\emptyset$ and $A\cup K$ is not connected.
\item If $N\subset\I$ is closed and nowhere dense, then $A\cup N$ is not connected.
\end{enumerate}
\end{lemma}

\begin{proof}
(b) $A\cup N$ is closed, nowhere dense and infinite, hence not an interval, hence not connected.

(a) Suppose that $A\cap\Int K\neq\emptyset$. This set is relatively open in the Cantor set $A$, hence infinite, so it
contains a point $a$ in the interior of one of the finitely many intervals $L\subset K$ given in
Section~\ref{sec:prelim}, on which $g=\varphi^{2r}$ is mixing, $r$ being the number of components of $K$. Choose
$0<\eps<\min\{\operatorname{dist}(a,\partial L),|L|/3\}$, and let $W,p$ be as in (S). Some $W'=\varphi^i(W)$ contains $a$;
it is a nondegenerate interval (since $\varphi$ is nonsingular) of length $<\eps$, so $W'\subset\Int L$, and
$\varphi^p(W')\subset W'$. Hence $g^{np}(W')\subset W'$ for all $n\ge1$. On the other hand, by the property of mixing
interval maps recalled in Section~\ref{sec:prelim}, $g^{np}(W')$ contains an interval of length at least $|L|/3>\eps$
for large $n$, a contradiction. So $A\cap\Int K=\emptyset$.

As $A$ is infinite and $\partial K$ finite, there is $a\in A\setminus K$; let $U$ be the component of $\I\setminus K$
containing $a$. Since $A\cap U$ is a nonempty relatively open subset of $A$, it contains two points $a_1<a_2$, and since
$A$ is nowhere dense there is $t\in(a_1,a_2)\setminus A$. Then $t\notin A\cup K$ and $A\cup K$ has points on both sides
of $t$, so it is not connected.
\end{proof}

\begin{lemma}[Products of interval attractors]\label{lem:product}
Let $A_j$ and $A_{j'}$ be of type \textup{(I)}. For Lebesgue almost every $(x,y)$ with $\omega_\varphi(x)=A_j$ and
$\omega_\varphi(f_2(y))=A_{j'}$, the set $\omega_{F^2}(x,y)$ has nonempty interior.
\end{lemma}

\begin{proof}
Write $A=A_j$, $B=B_{j'}=f_1(A_{j'})$, $\mu=\mu_j$, $\nu=\nu_{j'}$, and let $\Phi=F^2=\varphi\times\chi$. By
Lemma~\ref{lem:typical}(a), $\omega_\chi(y)=B$. Since $\omega_\varphi(x)=A$ has nonempty interior, the orbit of $x$
enters the forward invariant set $A$, and similarly the orbit of $y$ enters $B$; so $\Phi^n(x,y)\in A\times B$ for all
large $n$. The map $\Phi^n$ is nonsingular on $\I^2$ for each $n$, and $\omega_\Phi(x,y)=\omega_\Phi(\Phi^n(x,y))$. Hence
it suffices to prove that $\omega_\Phi(u,v)$ has nonempty interior for Lebesgue almost every $(u,v)\in A\times B$, and,
since $\mu\times\nu$ is equivalent to Lebesgue measure on $A\times B$ (Lemma~\ref{lem:typical}(c)), for
$(\mu\times\nu)$-almost every $(u,v)$.

Let $\ell,\ell'$ be the integers of Definition~\ref{def:regular}(I) and Lemma~\ref{lem:typical}(c) for $\mu$ and $\nu$,
and $n=\ell\ell'$. Each $\mu_t$ is $\varphi^n$-invariant and exact for $\varphi^n$ (a power of an exact
endomorphism is exact), hence weakly mixing; each $\nu_{t'}$ is $\chi^n$-invariant and exact, hence ergodic. The
product of a weakly mixing and an ergodic measure-preserving system is ergodic \cite[\S1.7]{Wa}, so
$\mu_t\times\nu_{t'}$ is ergodic for $\Phi^n$. By the ergodic theorem, for $(\mu_t\times\nu_{t'})$-almost every
$(u,v)$ the $\Phi^n$-orbit of $(u,v)$ visits every open set of positive measure infinitely often, so
\[
\omega_\Phi(u,v)\supset\omega_{\Phi^n}(u,v)\supset\operatorname{supp}(\mu_t\times\nu_{t'})=
\operatorname{supp}\mu_t\times\operatorname{supp}\nu_{t'},
\]
which has nonempty interior. Since $\mu\times\nu$ is the average of the measures $\mu_t\times\nu_{t'}$, the claim holds
$(\mu\times\nu)$-almost everywhere.
\end{proof}

\begin{proof}[Proof of Theorem~\ref{thm:E}]
Let $Z\subset Z_0$ be the full-measure set obtained by removing from $Z_0$ the null exceptional sets of
Lemma~\ref{lem:product} for all pairs $(j,j')$. Let $z=(x,y)\in Z$, $A=\omega_\varphi(x)$ and $A'=\omega_\varphi(f_2(y))$,
so that $\omega_\chi(y)=f_1(A')$ and, by Proposition~\ref{prop:BLC}(b), $\pi_1(\omega_F(z))=A\cup A'$.

If $A$ is of type (P), then $\omega_\varphi(x)$ is finite; if $A'$ is of type (P), then $\omega_\chi(y)=f_1(A')$ is finite.
In both cases (a) holds, by Proposition~\ref{prop:comb} and Remark~\ref{rem:swap}. If $A$ or $A'$ is a solenoid, then
$A\cup A'$ is not connected by Lemma~\ref{lem:solenoid}: if the other one is of type (P) or (S) use part (b), and if it is
of type (I) use part (a), since an $\omega$-limit set with nonempty interior is a transitive cycle of intervals. So (b) holds. Finally, if $A$ and $A'$ are
both of type (I), then $\Int\omega_F(z)\supset\Int\omega_{F^2}(z)\neq\emptyset$ by Lemma~\ref{lem:product}, and (c) holds.

Now Conjecture~\ref{conj:LM} holds at $z$: in case (a) by Theorem~\ref{thm:A}; in case (b) both sides of the
equivalence fail, since (C1) fails and the projection of a connected set is connected; and in case (c) by
Theorem~\ref{thm:LM}. If $\omega_F(z)$ is connected, then in case (a) it is a point or a lattice by Theorem~\ref{thm:C}(1)
and Remark~\ref{rem:swap}, case (b) is impossible, and in case (c) it is a finite union of periodic rectangles by
\cite{BLC}.
\end{proof}

\section{Examples}\label{sec:examples}

\begin{example}[Lattices; {\cite[Example~1]{LM}}]\label{ex:lattice}
Let $T$ be the tent map, $f_1=T$ and $f_2=\mathrm{id}$, so that $F(x,y)=(y,T(x))$ and $\varphi=\chi=T$. Let $x_0$
have a dense orbit and $y_0$ be periodic of period $N$, with orbit $P'$. Then
$\omega_F(x_0,y_0)=(\I\times P')\cup(P'\times\I)$. Here $\omega_\chi(y_0)$ is finite, and after the swap $\sigma$ this
is a lattice with $J=K=\I$. Theorem~\ref{thm:C} shows that all lattices, with arbitrary intervals $J,K$ and arbitrary
finite sets $P\subset K$, $Q\subset J$ of the same cardinality, occur.
\end{example}

\begin{example}[An alternating comb; {\cite[Fig.~36]{LM}}]\label{ex:alternating}
Let $f$ be the transitive, not totally transitive, map of \cite[\S5]{LM} with $f([0,\frac12])=[\frac12,1]$ and
$f([\frac12,1])=[0,\frac12]$, let $F(x,y)=(y,f(x))$, let $\{y_1,y_2\}=\{\frac16,\frac56\}$ be a $2$-cycle and
$x_0$ a point with dense orbit. Then
\[
\omega_F(x_0,y_1)=\bigl([0,\tfrac12]\times\{y_1\}\bigr)\cup\bigl(\{y_1\}\times[\tfrac12,1]\bigr)\cup
\bigl([\tfrac12,1]\times\{y_2\}\bigr)\cup\bigl(\{y_2\}\times[0,\tfrac12]\bigr),
\]
which is not connected although both projections are $[0,1]$ \cite{LM}. After the swap $\sigma$ this is an
alternating comb of order $2$, and Proposition~\ref{prop:alternating} explains why (C2) must fail.
\end{example}

\begin{example}[A robust lattice]\label{ex:mixed}
In the previous examples the lattice or comb is the $\omega$-limit set of a Lebesgue null set of points. This is not
always the case. Let $f_1=\mathrm{id}$ and let $g\colon\I\to\I$ be the piecewise linear map with
\[
g(t)=\begin{cases}\tfrac14+\tfrac t2, & 0\le t\le\tfrac12,\\[2pt] 2t-\tfrac12, & \tfrac12\le t\le\tfrac34,\\[2pt]
\tfrac52-2t, & \tfrac34\le t\le1,\end{cases}
\]
and $f_2=g$, so that $F(x,y)=(g(y),x)$ and $\varphi=\chi=g$. On $[\frac12,1]$, $g$ is a tent map, exact and preserving
normalized Lebesgue measure, and every $t\in[0,\frac12)$ converges to the fixed point $\frac12$. Thus $\varphi$ is
metrically regular, with attractors $\{\frac12\}$ of type (P) and $[\frac12,1]$ of type (I). For every $x<\frac12$ and
Lebesgue almost every $y>\frac12$ we have $P=\{\frac12\}$, $J=[\frac12,1]$, $K=g(J)=J$, $Q=\{\frac12\}$, and
\[
\omega_F(x,y)=\bigl(\{\tfrac12\}\times[\tfrac12,1]\bigr)\cup\bigl([\tfrac12,1]\times\{\tfrac12\}\bigr),
\]
a lattice of order $1$ in the shape of the letter L, which is therefore the $\omega$-limit set of a set of points of
measure $\frac14$. For almost every $(x,y)$ with $x,y>\frac12$ the $\omega$-limit set is $[\frac12,1]^2$, and for
$x,y<\frac12$ it is the fixed point $(\frac12,\frac12)$. The mechanism, coexistence of a periodic and a chaotic attractor, is familiar in duopoly models \cite{ABG};
here the periodic attractor lies on the boundary of the chaotic one, which is what makes $Q\subset J$ possible.
\end{example}

\begin{table}[t]
\centering
\renewcommand{\arraystretch}{1.2}
\begin{tabular}{@{}p{4.2cm}p{3.6cm}p{3.6cm}@{}}
\toprule
& Triangular maps \cite{Gu26} & Cournot maps (this paper)\\
\midrule
Dynamics in the fibre direction & non-autonomous & autonomous ($\chi$)\\
Types of sections over an $m$-cycle & may differ (Cantor absorption) & all of the same type (Prop.~\ref{prop:rigidity})\\
Connected, with $\pi_2$ connected & --- & lattices, all realized (Thms.~\ref{thm:B}, \ref{thm:C})\\
\bottomrule
\end{tabular}
\caption{$\omega$-limit sets over periodic orbits: triangular versus Cournot maps.}
\label{tab:contrast}
\end{table}

\section{Open problems}\label{sec:problems}

\begin{problem}\label{prob:conj}
Prove or disprove Conjecture~\ref{conj:LM} at the points not covered by Theorems~\ref{thm:A} and~\ref{thm:E}, for
instance when $\omega_{f_2\circ f_1}(x)$ is a Cantor set contained in a transitive cycle of intervals of
$f_2\circ f_1$, or when $\omega_{F^2}(z)$ is the graph of a semiconjugacy as in the circles of \cite{LM}.
\end{problem}

\begin{problem}\label{prob:H}
Is the letter H, $(\{0,1\}\times\I)\cup(\I\times\{\frac12\})$, an $\omega$-limit set of some Cournot map? By
Corollary~\ref{cor:D}, at such a point both $\omega_{f_2\circ f_1}(x)$ and $\omega_{f_1\circ f_2}(y)$ would have
to be infinite, and by Theorem~\ref{thm:E} the set of such points is null when $f_2\circ f_1$ is metrically regular.
\end{problem}

\begin{problem}\label{prob:general}
Describe the $\omega$-limit sets $\omega_F(z)$ with $\omega_{f_2\circ f_1}(x)$ finite when $\pi_2(\omega_F(z))$ is not
connected. By Proposition~\ref{prop:rigidity} the vertical sections are of a common type and $\chi(M_i)=M_{i+1}$;
which families $(M_i,N_i)$ arise?
\end{problem}

\begin{problem}\label{prob:unimodal}
Is every lattice an $\omega$-limit set of a Cournot map with unimodal (or even logistic) reaction functions? Is
$f_2\circ f_1$ metrically regular for Lebesgue almost every pair of logistic reaction functions?
\end{problem}

\begin{problem}\label{prob:ndim}
Extend the results to $n$-dimensional Cournot maps
\[
F(x_1,\dots,x_n)=\bigl(f_n(x_n),f_1(x_1),\dots,f_{n-1}(x_{n-1})\bigr)
\]
\cite{GR,LS,LC}. Proposition~\ref{prop:exact} applies to chains of any length, so the construction part should carry
over; the main task is the description of the corresponding combs.
\end{problem}



\begin{thebibliography}{99}

\bibitem{ABG}
A.~Agliari, G.-I. Bischi and L.~Gardini,
\emph{Some methods for the global analysis of dynamic games represented by iterated noninvertible maps},
in: Oligopoly Dynamics. Models and Tools (T.~Puu and I.~Sushko, eds.), Springer, Berlin, 2002, 31--83.

\bibitem{ABCP}
S.~J. Agronsky, A.~M. Bruckner, J.~G. Ceder and T.~L. Pearson,
\emph{The structure of $\omega$-limit sets for continuous functions},
Real Anal. Exchange \textbf{15} (1989/90), 483--510.

\bibitem{BCL01}
F.~Balibrea, J.~S. C\'anovas and A.~Linero,
\emph{New results on topological dynamics of antitriangular maps},
Appl. Gen. Topol. \textbf{2} (2001), 51--61.

\bibitem{BLC}
F.~Balibrea, A.~Linero and J.~S. C\'anovas,
\emph{On $\omega$-limit sets of antitriangular maps},
Topology Appl. \textbf{137} (2004), 13--19.

\bibitem{BM}
M.~Barge and J.~Martin,
\emph{Chaos, periodicity, and snakelike continua},
Trans. Amer. Math. Soc. \textbf{289} (1985), 355--365.

\bibitem{BL91}
A.~M. Blokh and M.~Yu. Lyubich,
\emph{Measurable dynamics of S-unimodal maps of the interval},
Ann. Sci. \'Ecole Norm. Sup. (4) \textbf{24} (1991), 545--573.

\bibitem{BC}
L.~S. Block and W.~A. Coppel,
\emph{Dynamics in One Dimension},
Lecture Notes in Math. \textbf{1513}, Springer, Berlin, 1992.

\bibitem{BPP}
P.~Brand\~ao, J.~Palis and V.~Pinheiro,
\emph{On the finiteness of attractors for one-dimensional maps with discontinuities},
Adv. Math. \textbf{389} (2021), Paper No. 107891, 49 pp.

\bibitem{BKNS}
H.~Bruin, G.~Keller, T.~Nowicki and S.~van Strien,
\emph{Wild Cantor attractors exist},
Ann. of Math. (2) \textbf{143} (1996), 97--130.

\bibitem{Co}
A.~Cournot,
\emph{Recherches sur les principes math\'ematiques de la th\'eorie des richesses},
Hachette, Paris, 1838.

\bibitem{DM}
R.~A. Dana and L.~Montrucchio,
\emph{Dynamic complexity in duopoly games},
J. Econom. Theory \textbf{40} (1986), 40--56.

\bibitem{Gu26}
J.~L. Garc\'{\i}a Guirao,
\emph{$\omega$-limit sets of triangular maps over periodic orbits of the base: chain transitivity, rigidity and
Cantor absorption},
preprint, 2026.

\bibitem{GR}
J.~L. Garc\'{\i}a Guirao and R.~Garc\'{\i}a Rubio,
\emph{Detecting simple dynamics in Cournot-like models},
J. Comput. Appl. Math. \textbf{233} (2009), 1091--1095.

\bibitem{KS97}
S.~F. Kolyada and \v{L}.~Snoha,
\emph{Some aspects of topological transitivity---a survey},
Grazer Math. Ber. \textbf{334} (1997), 3--35.

\bibitem{KSvS}
O.~Kozlovski, W.~Shen and S.~van Strien,
\emph{Density of hyperbolicity in dimension one},
Ann. of Math. (2) \textbf{166} (2007), 145--182.

\bibitem{Le81}
F.~Ledrappier,
\emph{Some properties of absolutely continuous invariant measures on an interval},
Ergodic Theory Dynam. Systems \textbf{1} (1981), 77--93.

\bibitem{LM}
A.~Linero Bas and M.~Mu\~noz Guillermo,
\emph{A full description of $\omega$-limit sets of Cournot maps having non-empty interior and some economic
applications},
Mathematics \textbf{9} (2021), no.~4, 452.

\bibitem{LS}
A.~Linero Bas and G.~Soler L\'opez,
\emph{A note on the dynamics of cyclically permuted direct product maps},
Topology Appl. \textbf{203} (2016), 147--158.

\bibitem{LC}
Z.~Luo and Y.~Chen,
\emph{On $\omega$-limit sets of antitriangular maps},
Nat. Sci. J. Xiangtan Univ. \textbf{32} (2010), no.~4, 6--9.

\bibitem{Pu}
T.~Puu,
\emph{Nonlinear Economic Dynamics}, 4th ed.,
Springer, Berlin, 1997.

\bibitem{Ru}
S.~Ruette,
\emph{Chaos on the Interval},
University Lecture Series \textbf{67}, Amer. Math. Soc., Providence, RI, 2017.

\bibitem{Sh}
A.~N. Sharkovsky,
\emph{Continuous mapping on the set of limit points of an iteration sequence},
Ukrain. Mat. Zh. \textbf{18} (1966), 127--130 (in Russian).

\bibitem{vSV}
S.~van Strien and E.~Vargas,
\emph{Real bounds, ergodicity and negative Schwarzian for multimodal maps},
J. Amer. Math. Soc. \textbf{17} (2004), 749--782; Erratum, ibid. \textbf{20} (2007), 267--268.

\bibitem{Wa}
P.~Walters,
\emph{An Introduction to Ergodic Theory},
Graduate Texts in Math. \textbf{79}, Springer, New York, 1982.

\end{thebibliography}
\end{document}